\documentclass[10pt]{amsart}
\usepackage{tikz-cd}
\usepackage{tikz}
\usetikzlibrary{decorations.pathreplacing,arrows.meta}
\usepackage{amssymb}
\usepackage[foot]{amsaddr}
\usepackage{todonotes}
\usepackage{graphicx}
\usepackage{hyperref}
\usepackage{standalone}

\makeatletter
    
    \@addtoreset{equation}{section}
  \makeatother

\newtheorem{thm}{Theorem}[section]
\newtheorem{coro}[thm]{Corollary}
\newtheorem{prop}[thm]{Proposition}
\newtheorem{lem}[thm]{Lemma}
\theoremstyle{definition}
\newtheorem{defi}[thm]{Definition}
\newtheorem{rmk}[thm]{Remark}
\newtheorem{ex}[thm]{Example}

\newtheorem*{notation*}{Notation}

\newcommand{\N}{\mathbb N}
\newcommand{\R}{\mathbb R}
\newcommand{\T}{\mathcal{T}(S)}

\newcommand{\curr}{\mathrm{Curr}(S)}
\newcommand{\ext}{\mathrm{Ext}_{X}}
\newcommand{\sqext}{\sqrt{\mathrm{Ext}_{X}}}

\title{On the extremal length of the hyperbolic metric}
\author{Hidetoshi Masai}
\address{Humanities and Sciences/Museum Careers, Musashino Art University
Bldg. 12, 1-736 Ogawa-cho, Kodaira-shi, Tokyo 187-8505}
\email{hmasai@musabi.ac.jp}
\thanks{The work of the author was partially supported by JSPS KAKENHI Grant Number 23K03085.}
\begin{document}
\begin{abstract}
For any closed hyperbolic Riemann surface \(X\), we show that the extremal length of the Liouville current is determined solely by the topology of \(X\).
This confirms a conjecture of Mart\'inez-Granado and Thurston.
We also obtain an upper bound, depending only on \(X\), for the diameter of extremal metrics on \( X \) with area one.
\end{abstract}
\maketitle
\section{Introduction}
For an orientable closed surface \( S \) of genus \( g \geq 2 \), 
there is a correspondence between its hyperbolic structures and Riemann surface structures. As a result, the Teichm\"uller space \( \T \) can be viewed either as the space of marked hyperbolic surfaces or marked Riemann surfaces. For any closed curve \( \gamma \subset S \), one can consider both its hyperbolic length and extremal length, with the latter offering a natural notion of length in the context of Riemann surfaces.

Bonahon introduced the notion of geodesic currents \cite{Bonahon}.
By assigning the so-called Liouville current $L_{X}$ to each point of $X\in\T$, Teichm\"uller space $\T$ naturally embeds into the space of geodesic currents, which we denote by $\curr$.
In this paper, a {\em multi-curve} means a family of (simple or non-simple) closed curves.
Every multi-curve naturally corresponds to a current, and hence we may consider weights on it.
Let us summarize the work of Bonahon.
\begin{thm}[Bonahon \cite{Bonahon}]\label{thm.Bonahon}
The following statements hold.
\begin{enumerate}
    \item The set of weighted multi-curves is dense in \( \curr \).
    \item The geometric intersection number $i(\cdot, \cdot)$ of closed curves extends continuously to \( \curr\times \curr \).
    \item For a closed curve \( \gamma \) and a Liouville current \( L_{X} \), we have  
          \(
          i(L_{X}, \gamma) = \ell_{X}(\gamma),
          \)  
          where \( \ell_{X}(\gamma) \) is the hyperbolic length of \( \gamma \).
      \item $i(L_{X},L_{X})=\pi\mathrm{Area}(X)/2$ for any $X\in\T$, 
      where $\mathrm{Area}(X)=2\pi|\chi(S)|$ is the hyperbolic area of $X$, and $\chi(S)$ is the Euler characteristic of $S$ which depends only on the topology of $S$.
\end{enumerate}
\end{thm}

See, for example, \cite{GT, zbMATH07539256} for beautiful applications of geodesic currents.
Mart\'inez-Granado and Thurston \cite{GT} observed that many ``length functions'',
which measure the length of closed curves on the surface $S$, extend continuously to the space $\curr$.
In particular, they showed that for any $X\in\T$, the square root of the extremal length function $\sqext(\cdot)$ gives a continuous function $\sqext:\curr\to\R$.
The notion of extremal length is known to be important in Teichm\"uller theory, see for example \cite{Gardiner-Masur, Kerckhoff, zbMATH06384016, sibling} and references therein.
This note aims to prove the following, which contains a conjecture of Mart\'inez-Granado and Thurston.

\begin{thm}[{c.f. \cite[Conjecture 4.18]{GT}}]\label{thm.main}
For any $X\in\T$ and $\mu\in\curr$, 
we have
\begin{equation}
\sqext(\mu) = \sup_{\rho}\frac{\ell_{\rho}(\mu)}{\sqrt{\mathrm{Area}(\rho)}}\label{eq.main}
\end{equation}
where $\rho$ runs over all allowable conformal metrics on $X$ (see Definition \ref{defi.ext} and Proposition \ref{prop.confext} for the details).
In particular, we have
\[
\ext(L_X) = \frac{~\pi^{2}}{4} \operatorname{Area}(X) = \frac{~\pi^{3}}{2}|\chi(S)|.
\]
\end{thm}
See Remark \ref{rmk.subtle} for the difficulty of the statement of Theorem \ref{thm.main}.
\begin{rmk}\label{rmk.hyp}
A conformal metric that attains the supremum in \eqref{eq.main} is called the \emph{extremal metric} (see Appendix \ref{sec.extremal metric} for more details). 

In general extremal metrics are very mysterious.
Theorem \ref{thm.main} implies that 
the hyperbolic metric \(\rho_X\) is an extremal metric for the Liouville current \(L_X\).
As a comparison, we prove that the hyperbolic metric cannot be extremal for any weighted multi-curve in Corollary \ref{coro.hyp}.

The statement of Theorem \ref{thm.main} differs from the conjecture in \cite{Liu--ext} by a factor of 16. This is due to the difference in normalization in \cite[Proposition 2.5]{Liu--ext}.
\end{rmk}

In the appendix, assuming their existence, we discuss extremal metrics for weighted multi-curves. 
We remark that this discussion is independent of the previous results.
The following theorem may be of independent interest.
\begin{thm}\label{thm.extbound}
Let $X\in\T$. Then there exists $D=D(X)>0$ such that for any weighted multi-curve $c$, the extremal metric of area $1$ for $c$ on $X$ has diameter at most $D$.
\end{thm}

\section{Geodesic flow and conformal structures}\label{sec.2}
Let $X\in\T$, and we first regard $X$ as a hyperbolic surface.
As is well-known to be an idea of Thurston (see \cite[p.151]{Bonahon} for details),
the Liouville current $L_{X}$ is obtained as the limit of ``random'' closed geodesics as follows.
Let \( T_1 X \) denote the unit tangent bundle of \( X \), and let \( \mathcal{L}_X \) denote the Liouville measure on \( T_1 X \), which is locally the product of the hyperbolic area measure and the angular measure.
Pick $v\in T^{1}X$,
and consider the geodesic flow trajectory $\varphi_{t}(v)$ on $X$ given by $v$.
Let $D_{X}>0$ be a fixed constant so that for any $t$, we may connect $\varphi_{t}(v)$ and $\varphi_{0}(v)$ by a path of length less than $D_{X}$.
This procedure gives a closed curve $g_{t}(v)\subset X$.
Then the Liouville current $L_{X}$ is characterized as (\cite[p.151]{Bonahon}),

\begin{equation}
\lim_{t\to\infty}\frac{i(L_{X},L_{X})}{i(L_{X},g_{t}(v))}\cdot g_{t}(v) = L_{X}.
\end{equation}

For later convenience, let
\begin{equation}
G_{t}(v):=\frac{i(L_{X},L_{X})}{i(L_{X},g_{t}(v))}\cdot g_{t}(v) \in\curr.
\end{equation}

One may check the normalization constant by considering $i(L_{X},G_{t}(v))$ compared with $i(L_{X},L_{X})$.
Notice the following.
\begin{prop}\label{prop.ergodic}
There exists $C=C(X,v)>0$ which is independent of $t$ such that we have
\begin{equation}
|t-i(L_{X},g_{t}(v))|<C
\end{equation}
for any $t$, where $t$ is the length of the flow trajectory from $\varphi_{t}(0)$ to $\varphi_{t}(v)$.
\end{prop}
\begin{proof}
The curve $g_{t}(v)$ is a concatenation of paths of length at most $D_{X}$ and a geodesic flow trajectory. Hence, if $|t|$ is large enough, $g_{t}(v)$ is a quasi-geodesic. In the universal covering, the limit points of the lift of $g_{t}(v)$ converge to those of the geodesic flow trajectory $\varphi_{t}(v)$ as $t\to\infty$. Hence, we see that for large enough $t$, a large portion of $g_{t}(v)$ fellow travels with $\varphi_{t}(v)$.
\end{proof}

Now we regard $X\in\T$ as a Riemann surface.
\begin{defi}\label{defi.ext}
A metric $\rho(z)|dz|$ on $X$ is called {\em an allowable conformal} metric if
$\rho$ is Borel measurable, non-negative, and locally $L_{2}$, and its area defined by 
\begin{equation}
\displaystyle \mathrm{Area}(\rho) := \int_{X}\rho^{2}dxdy
\end{equation}
is neither $0$ nor $\infty$.
Let $\Gamma = \{\gamma_{1},\dots, \gamma_{n}\}$ be a family of closed curves and arcs on $X$.
The {\em extremal length} of $\Gamma$ is defined as 
\begin{equation}
\ext(\Gamma):=\sup_{\rho}\frac{\ell_{\rho}(\Gamma)^{2}}{\mathrm{Area}(\rho)}\label{eq.defext}
\end{equation}
where the supremum is taken over all the allowable conformal metrics on $X$ and
\begin{enumerate}
\item[(i)] $\displaystyle L_{\rho}(\gamma):=\int_{\gamma}\rho|dz|$ is the $\rho$-length of a path $\gamma$,
\item[(ii)] $\displaystyle\ell_{\rho}(\Gamma):=\sum_{i=1}^{n}\inf_{\gamma'_{i}}L_{\rho}(\gamma'_{i})$ 
where the infimum is taken over all $\gamma'_{i}$ homotopic to $\gamma_{i}$ relative to the boundary.
\end{enumerate}
Using the $\rho$-length function, we define $\rho$-distance $d_{\rho}(\cdot,\cdot):X\times X\to\R_{\geq 0}$ by
$$
d_{\rho}(x,y) = \inf_{\gamma}\ell_{\rho}(\gamma)
$$
where the infimum is taken over all the arcs connecting $x$ and $y$ in $X$.
\end{defi}
Let us summarize the work of \cite{GT}.
\begin{prop}[{\cite[Section 4.3, Section 4.8]{GT}}]\label{prop.confext}
For any conformal metric $\rho(z)|dz|$ on $X$, the length function $\ell_{\rho}(\cdot)$ extends continuously to $\curr$.
The square root of the extremal length function $\sqext(\cdot)$ also extends continuously to $\curr$.
\end{prop}

We first prove an easy consequence of the definition of extremal length.
\begin{prop}\label{prop.lower}
For any $X$ and $\mu\in\curr$ we have
\begin{equation}
\sup_{\rho}\frac{\ell_{\rho}(\mu)}{\sqrt{\mathrm{Area}(\rho)}}\leq \sqext(\mu)\label{eq.lower1}
\end{equation}
In particular, we have 
\begin{equation}
\frac{\pi}{2}\sqrt{\mathrm{Area}(X)}=\frac{i(L_{X},L_{X})}{\sqrt{\mathrm{Area}(X)}}
\leq {\sqext(L_{X})}.\label{eq.ext}
\end{equation}
\end{prop}
\begin{proof}
By the definition of extremal length, for any weighted multi-curve $c$ and any conformal metric $\rho$, we have
$$\frac{\ell_{\rho}(c)}{\sqrt{\mathrm{Area}(\rho)}}\leq \sqext(c).$$

Since the weighted multi-curves are dense in $\curr$ (Theorem \ref{thm.Bonahon}), 
and the maps 
$\ell_{\rho}(\cdot), \sqext(\cdot):\curr\to\R_{\geq 0}$ are continuous (Proposition \ref{prop.confext}), we have \eqref{eq.lower1}.

Let $\rho_{X}$ denote the hyperbolic metric on $X$.
Since $\rho_{X}$ is one of the conformal metrics, we have
\begin{equation}
\frac{\pi}{2}\sqrt{\mathrm{Area}(X) }
=\frac{i(L_{X},L_{X})}{\sqrt{\mathrm{Area}(X)}}
= \frac{\ell_{\rho_{X}}(L_{X})}{\sqrt{\mathrm{Area}(X)}}
\leq \sqext(L_{X})\label{eq.hyp}
\end{equation}
\end{proof}

We now focus on the Liouville current $L_{X}$.
Recall the classical work of Hopf.
Although the original statement of Hopf is for the unit tangent bundle $T_{1}X$,
we state here for $X$ as conformal metrics are independent of angles.
\begin{thm}[{\cite[FIRST THEOREM]{Hopf}}]\label{thm.Hopf-original}
For $X\in\T$, the geodesic flow is ergodic.
In other words, if 
\( f(\cdot) \) and \( g(\cdot) > 0 \) are integrable with respect to 
the hyperbolic metric $\rho^{2}_{X}dxdy$
then
\[
\lim_{T \to \infty} \frac{\displaystyle \int_0^{T} f(\varphi_{t}(v))\, dt}{\displaystyle \int_0^{T} g(\varphi_{t}(v))\, dt}
= \frac{\displaystyle \int_{X} f\, \rho^{2}_{X}dxdy}{\displaystyle \int_{X} g\, \rho^{2}_{X}dxdy}
\]
holds for $\mathcal{L}_{X}$-almost every \( v \in T^{1}X \).
The same holds for the limit as \( T \to -\infty \).
\end{thm}
One key step to obtain the inverse inequality to \eqref{eq.ext} is the following.
\begin{thm}\label{thm.Hopf}
For any conformal structure $\rho(z)|dz|$ on $X$, we have
\begin{equation}
\lim_{T\to\infty}\frac{L_{\rho}(G_{T}(v))}{\sqrt{\mathrm{Area}(\rho)}}\leq \frac{\pi}{2}\sqrt{\mathrm{Area}(X)}.\label{eq.Hopf1}
\end{equation}
In particular, we have
\begin{equation}
\frac{\ell_{\rho}(L_{X})^{2}}{\mathrm{Area}(\rho)}\leq \frac{~\pi^{2}}{4}\mathrm{Area}(X).\label{eq.Hopf2}
\end{equation}
\end{thm}
\begin{proof}
For any allowable conformal metric $\rho$ on $X$, 
by Theorem \ref{thm.Hopf-original}  applied for $f=\rho/\rho_{X}$ and $g=1$, we have
\begin{equation}
\lim_{T\to\infty}\frac{1}{T}\int_{0}^{T}\frac{\rho(\varphi_{t}(v))}{\rho_{X}(\varphi_{t}(v))}dt
= \frac{1}{\mathrm{Area(X)}}\int_{X}\frac{\rho}{\rho_{X}}\cdot\rho_{X}^{2}dxdy.\label{eq.Birkhoff}
\end{equation}
By Theorem \ref{thm.Bonahon} and Proposition \ref{prop.ergodic}, 
we see that $|i(L_{X},g_{T}(v))-T|\leq C$.
Note that the integral $\int_{0}^{T}\frac{\rho(\varphi_{t}(v))}{\rho_{X}(\varphi_{t}(v))}dt$ equals $\rho$-length of the geodesic flow trajectory from $\varphi_{0}(v)$ to $\varphi_{T}(v)$, and hence there exists $C'=C'(\rho)\geq 0$ such that 
$$\left|\int_{0}^{T}\frac{\rho(\varphi_{t}(v))}{\rho_{X}(\varphi_{t}(v))}dt - L_{\rho}(g_{T}(v))\right|\leq C'.$$
Hence, the left-hand side of \eqref{eq.Birkhoff} is equal to 
\begin{equation}
\lim_{T\to\infty}\frac{1}{i(L_{X},g_{T}(v))}\cdot L_{\rho}(g_{T}(v)).
\end{equation}

By the Cauchy-Schwarz inequality, the right-hand side of \eqref{eq.Birkhoff} is
\begin{align}
\frac{1}{\mathrm{Area(X)}}\int_{X}\rho\cdot\rho_{X}dxdy
\leq \frac{1}{\mathrm{Area(X)}}\sqrt{\int_{X}\rho^{2}dxdy\cdot \int_{X}\rho_{X}^{2}dxdy} = \frac{\sqrt{\mathrm{Area}(\rho)}}{\sqrt{\mathrm{Area}(X)}}
\end{align}
Hence by \eqref{eq.Birkhoff}, we have
\begin{align}
\nonumber&\lim_{T\to\infty}\frac{1}{i(L_{X},g_{T}(v))}\cdot \frac{L_{\rho}(g_{T}(v))}{\sqrt{\mathrm{Area}(\rho)}}\leq \frac{1}{\sqrt{\mathrm{Area}(X)}}\\\nonumber
\iff &\lim_{T\to\infty}\frac{i(L_{X},L_{X})}{i(L_{X},g_{T}(v))}\cdot \frac{L_{\rho}(g_{T}(v))}{\sqrt{\mathrm{Area}(\rho)}}\leq \frac{i(L_{X},L_{X})}{\sqrt{\mathrm{Area}(X)}}\\\nonumber
\iff &\lim_{T\to\infty}\frac{L_{\rho}(G_{T}(v))}{\sqrt{\mathrm{Area}(\rho)}}\leq \frac{\pi}{2}\sqrt{\mathrm{Area}(X)}.
\label{eq.est}
\end{align}
This completes the proof of the inequality \eqref{eq.Hopf1}.

The inequality \eqref{eq.Hopf2} follows as $\ell_{\rho}$ extends continuously to $\curr$ (Proposition \ref{prop.confext}), 
and $\ell_{\rho}(G_{T}(v))\leq L_{\rho}(G_{T}(v))$.
\end{proof}
\begin{rmk}\label{rmk.subtle}
For any $\mu\in\curr$, if we had
\begin{equation}
\sqext(\mu) = \sup_{\rho}\frac{\ell_{\rho}(\mu)}{\sqrt{\mathrm{Area}(\rho)}}\label{eq.hopeext}
\end{equation}
then Proposition \ref{prop.lower} and Theorem \ref{thm.Hopf} would give Theorem \ref{thm.main}.

However, since the pointwise supremum of a family of continuous functions is generally only lower semicontinuous, equation \eqref{eq.hopeext} requires careful handling.
The rest of the paper will be devoted to this issue, verifying that \eqref{eq.hopeext} is indeed valid.
\end{rmk}
\section{Upper bound}
To prove Theorem \ref{thm.main}, we briefly recall the work of Mart\'inez-Granado and Thurston, readers are referred to \cite{GT} for details.
\subsection{Return trajectories}
We will follow the notation of \cite{GT} as closely as possible.
Let $Y:=T_{1}X$ denote the unit tangent bundle of $X$.
The 3-manifold $Y$ admits a natural geodesic flow $\phi_{t}$ via the hyperbolic structure on $X$.
In \cite[Section 8]{GT}, they showed the existence of a so-called global cross-section $\tau\subset Y$.
The $\tau$ satisfies
\begin{itemize}
\item $\tau$ is a compact smooth codimension $1$ submanifold-with-boundary 
that is smoothly transverse to the foliation of $Y$ given by $\phi_{t}$.
\item for any $y\in Y$, there exist $s<0<t$  such that $\phi_{s}(y)\in\tau, \phi_{t}(y)\in\tau$.
\item $\tau$ is the image of an immersion of a disk.
\end{itemize}
We then have the first return map $p:\tau\to\tau$.
In short, \( \tau \) is constructed as a “wedge set” from a closed curve \( \delta\subset X \) and a small interval \( I \subset \delta \) by:
\begin{equation}
\tau = \{(x,v)\in T^{1}X\mid 
x\in\delta\setminus I, ~|\angle(T_{x}\delta, v)-\pi/{2}|\leq\pi/{6}\}.
\end{equation}
(Angles $\angle(v,w)$ are measured by the counterclockwise rotation from $v$ to $w$).

However, the continuity of $p:\tau\to\tau$ breaks down along the boundary $\partial\tau$.
To overcome this difficulty, nested global cross-sections
$\tau_{0}\Subset \tau\subset\tau'$ are considered, 
where $\tau_{0}\Subset\tau$ means that $\partial\tau_{0}$ is contained in the interior of $\tau$.
Then a continuous bump function $\psi:\tau \to [0,1]$ (\cite[Section 7]{GT}) with the property that $\psi$ is $1$ on $\tau_{0}$ and $0$ on an open neighborhood of $\partial \tau$ is considered.

Given a topological space $M$, 
let $\R_{1}M$ be the space of Borel measures with finite support and total mass $1$ on $M$.
Using $\psi$, a map $P:\tau\to\R_{1}\tau$ is defined inductively by 
\[
P(x) :=
\begin{cases}
\,p(x) & p(x)\in \tau_{0} \\[6pt]
\,\psi(p(x))\cdot p(x) + (1-{\psi}(p(x))\cdot P(p(x)) & p(x)\in \tau \setminus \tau_{0}\,.
\end{cases}
\]

Then it is shown that $P$ is continuous \cite[Proposition 7.7]{GT}.
The return trajectory is defined as follows.
\begin{defi}[{\cite[Definition 7.17]{GT}}]
Let $\phi_t$ be the geodesic flow on $Y$ and let $\tau$ be a global cross-section contained in a larger compact simply connected cross-section $\tau'$. 
Fix a basepoint $*\in\tau'$. 
For $x\in\tau$, define the \emph{return trajectory} $m(x)\in\pi_1(Y,*)$ by taking the homotopy class of a path that runs in $\tau'$ from $*$ to $x$, along the flow trajectory from $x$ to $p(x)$, and then in $\tau'$ from $p(x)$ back to $*$. 
Since $\tau$ is the image of an immersion of a disc, 
$m(x)$ is independent of the choice of path.
\end{defi}

\begin{defi}[{\cite[Definition 7.19]{GT}}]
The \emph{homotopy return map} is the map 
$q : \tau \to \tau \times \pi_1(Y,*)$ defined by
\[
q(x) := (p(x), m(x)).
\]

We can iterate $q$ by inductively defining $q^{n+1}$ to be the composition
\[
\tau \xrightarrow{q^n} \tau \times \pi_1(Y) \xrightarrow{q \times \mathrm{id}} \tau \times \pi_1(Y) \times \pi_1(Y) \xrightarrow{(x,g,h)\mapsto(x,hg)} \tau \times \pi_1(Y).
\]

Define $m^n(x)\in \pi_1(Y,*)$ to be the second component of $q^n(x)$.
\end{defi}
\begin{defi}[{\cite[Definition 7.21]{GT}}]
The \emph{smeared homotopy return map} 
\[
Q : \tau \to \mathbb{R}_1(\tau \times \pi_1(Y,*))
\]
is defined by
\[
Q(x) := 
\begin{cases}
\,q(x) & p(x)\in \tau_0 \\[6pt]
\,\psi(p(x)) \cdot q(x) + \left(1-\psi(p(x))\right) \cdot L_{m(x)} Q(p(x)) & p(x)\in \tau - \tau_0
\end{cases}
\]
where $L_g$ is left translation by $g \in \pi_1(Y,*)$:
\[
L_g\left(\sum_i a_i(x_i,h_i)\right) := \sum_i a_i(x_i,gh_i).
\]
\end{defi}
Iteration of $Q$ is also well defined, see \cite{GT} for the details.

\begin{defi}[{\cite[Definition 7.22]{GT}}]
We define the \emph{smeared $n$-th return trajectory} 
\[
M^n : \tau \to \mathbb{R}_1\pi_1(Y)
\]
to be the composition
\[
\tau \xrightarrow{Q^n} \mathbb{R}_1(\tau\times\pi_1(Y)) \longrightarrow \mathbb{R}_1\pi_1(Y)
\]
where at the second step we lift the projection on the second component to act on weighted objects (see \cite[Definition 7.8]{GT}).
Let $\Lambda(n,\tau)$ be the set of curves that appear with non-zero coefficient in $M^n(x)$ for some $x\in\tau$.
\end{defi}
By appealing to the compactness of $Y$ and the length bound for curves in $\Lambda(n,\tau)$, it is proved in \cite[Lemma 7.23]{GT} that 
$\Lambda(n,\tau)$ is finite.
It is also proved that $Q^{n}$, $M^{n}$ are continuous in \cite[Lemma 7.24]{GT}.

The curves $M^{k}(x)$ project to a weighted multi-curve on $X$.
Thus, we obtain
$$[M^{k}(\cdot)]: \tau \to \{\text{weighted multi-curves on $X$}\}\subset\curr.$$

Note that each geodesic current $\mu\in\curr$ is invariant under geodesic flow and hence descends to a measure on a global cross-section $\tau$ of the geodesic flow $\phi_{t}$. 
We use the same notation $\mu$ for the measure on $\tau$. 

By the finiteness of $\Lambda(n,\tau)$, we see that
$$\int_{\tau} [M^{n+k}(x)]\psi(x)\mu(x)$$
is a weighted multi-curve in $\curr$.
The following join lemma is very useful.
A smoothing is a local operation on intersections of curves:~~~
\begin{tikzpicture}[scale=0.15]
\begin{scope}[shift={(-3,0)}]
    \draw[-{Latex[scale=0.5]}] (-1,-1) -- (1,1);
    \draw[-{Latex[scale=0.5]}] (-1,1) -- (1,-1);
\end{scope}

\draw[thick,-{Latex[scale=0.5]}] (-0.8,0) -- (0.8,0);

\begin{scope}[shift={(3,0)}]
    \draw[-{Latex[scale=0.5]}] (-1,1) .. controls (0,0.2) .. (1,1);
    \draw[-{Latex[scale=0.5]}] (-1,-1) .. controls (0,-0.2) .. (1,-1);
\end{scope}
\end{tikzpicture}
.
\begin{lem}[{\cite[Lemma 9.2 (Smeared join lemma)]{GT}}]\label{lem.join}
Let \(\tau\) be a global cross-section. There is a curve \(K_\tau\) and an integer \(w_\tau\) such that for large enough \(n, k \geq 0\), we have, for all \(x \in \tau\),
\begin{itemize}
  \item[(a)] \([M^{n}(x)] \cup [M^{k}(P^{n}(x))] \cup K_\tau \searrow_{w_\tau} [M^{n+k}(x)]\),
  \item[(b)] \([M^{n+k}(x)] \cup K_\tau \searrow_{w_\tau} [M^{n}(x)] \cup [M^{k}(P^{n}(x))]\).
\end{itemize}
where $\searrow_{w_\tau}$ means that the right-hand side is obtained from the left-hand side by smoothing $w_{\tau}$ crossings.
\end{lem}

Let  $f$ be a length function defined on all closed curves on $X$ satisfying certain conditions (see \cite[Theorem A]{GT} for the details).
The same argument as in \cite[Proposition 9.4]{GT} applies to our situation. 
Note that we do not consider ``quasi-smoothings'' here (see \cite{GT} for details), and hence inequalities are simpler than those in \cite{GT}.
For sufficiently large $n, k$:
\begin{align}
f^{n+k}(\mu) 
:= & f\left(\int_{\tau} [M^{n+k}(x)] \psi(x)\mu(x)\right)\label{eq.4.1}
\\
\le & f\left(\int_{\tau}[M^n(x)] \psi(x)\mu(x)\right) + f\left(\int_{\tau}[M^{k}(P^n(x))] \psi(x)\mu(x)\right) +f(K_\tau)A_\tau(\mu)
\label{eq.4.2}\\
= & f^n(\mu)+f\left(\int_{\tau}[M^{k}(x)]P^n_{*}( \psi(x)\mu(x))\right)+f(K_\tau)A_\tau(\mu)
\label{eq.4.3}\\
= & f^n(\mu)+f^k(\mu)+f(K_\tau)A_\tau(\mu)\label{eq.4.4}
\end{align}
where $A_{\tau}(\mu) = \int_{\tau}\psi(x)\mu(x)$.
The inequality \eqref{eq.4.2} is due to Lemma \ref{lem.join} (a), together with the fact that the length decreases after smoothing, and convex union property of $f$, namely $f(\alpha\cup\beta)\leq f(\alpha)+f(\beta)$, which is satisfied when $f(\cdot) = \sqext(\cdot)$ (\cite[Lemma 4.17.]{GT}).
The equality \eqref{eq.4.4} follows from the invariance of $\psi\mu$ under $P$ \cite[Proposition 7.16]{GT}, namely
{
\begin{equation}
\int_{\tau}[M^{k}(P^n(x))]\psi(x)\mu(x) = \int_{\tau}[M^{k}(x)]P^n_{*}(\psi(x)\mu(x))= \int_{\tau}[M^{k}(x)]\psi(x)\mu(x)=: \Lambda(k,\mu)\label{eq.Pinv}
\end{equation}
}
as a weighted multi-curve.

The equation \eqref{eq.4.4} corresponds to the subadditivity.
By using the subadditivity, it is proved that : 
\begin{thm}[{\cite[Proposition 9.6, Proposition 10.8, Proposition 11.5 and Theorem 13.1]{GT}}]\label{thm.MGT-main}
Suppose that $f$ satisfies certain natural conditions (see \cite{GT} for the details, the conditions are satisfied by $\sqext$ and $\ell_{\rho}$ for any conformal metric $\rho$).
Then the limit
$$
f_{\tau}(\mu) = \lim_{n\to\infty}\frac{f^{n}(\mu)}{n}
$$
exists and we have $f_{\tau}(\mu) = f(\mu)$ when $\mu$ corresponds to a weighted closed curve.
The function $f_{\tau}:\curr\to\R$ is the continuous extension of $f$ to $\curr$.
\end{thm}

\subsection{Proof of Theorem \ref{thm.main}}
Now we consider the case where $f=\sqext$.
Let $\mu$ be a geodesic current and let $\Lambda_{n}:=\Lambda(n, \mu)/n$.
Then we have 
\begin{equation}
\sqext(\mu) = \lim_{n\to\infty}\sqext(\Lambda_{n}), \text{ and } 
\ell_{\rho}(\mu) = \lim_{n\to\infty}\ell_{\rho}(\Lambda_{n}).\label{eq.Lambda}
\end{equation}
by Theorem \ref{thm.MGT-main}.
Using the definition of extremal length,
for any $\epsilon>0$, we have a conformal metric $\rho_{n}$ such that
$$
\sqext\left( \Lambda_{n}\right) \leq 
\ell_{ \rho_{n}}\left( \Lambda_{n}\right)+\epsilon.
$$
By the definition of extremal length, we have for $K_{\tau}$ in Lemma \ref{lem.join}
\begin{equation}
\ell_{\rho_{n}}(K_{\tau})\leq \sqext(K_{\tau})
\end{equation}
for any $n\in\N$.

Let $A_{\tau}:=A_\tau(\mu)$. By Lemma \ref{lem.join} with part (a) applied at \eqref{eq.3.6} and part (b) applied at \eqref{eq.3.8}, each used \( I \) times, and noticing the fact that $$\ell_{\rho}(\alpha\cup\beta) = \ell_{\rho}(\alpha)+\ell_{\rho}(\beta)$$ for any conformal metric $\rho$, we obtain
\begin{align}
\sqext(\Lambda_{nI}) &= \sqext\left(\int_{\tau}\frac{[M^{nI}(x)]}{nI}\psi(x)\mu(x)\right) 
\nonumber\\
&\leq \sum_{i=0}^{I-1}\sqext\left(\int_{\tau}\frac{[M^{n}(P^{ni}(x))] }{nI}\psi(x)\mu(x)
\right) + \frac{A_{\tau}\sqext(K_{\tau})\cdot I}{nI}\label{eq.3.6}\\
&\leq
\sum_{i=0}^{I-1}\left\{\ell_{\rho_{n}}\left(\int_{\tau}\frac{[M^{n}(P^{ni}(x))] }{nI}\psi(x)\mu(x)\right)+ \frac{\epsilon}{I}\right\} +\frac{A_{\tau}\sqext(K_{\tau})}{n}\label{eq.Pinv-used}\\
&\leq \ell_{ \rho_{n}}\left(\int_{\tau}\frac{[M^{nI}(x)]}{nI} \psi(x)\mu(x)\right)  +  
\epsilon+\frac{A_{\tau}\sqext(K_{\tau})}{n} + \frac{A_{\tau}\ell_{\rho_{n}}(K_{\tau})}{n}\label{eq.3.8}\\
&\leq \ell_{ \rho_{n}}\left(\Lambda_{nI}\right) +\epsilon+  
\frac{2A_{\tau}\sqext(K_{\tau})}{n} \label{eq.4.8}
\end{align}
where the $P$ invariance of measures \eqref{eq.Pinv} is used to get \eqref{eq.Pinv-used}.
Then by \eqref{eq.Lambda},
\begin{align}
\sqext(\mu)
&=\lim_{I\to\infty}\sqext(\Lambda_{nI})\\
&\leq \lim_{I\to\infty}\ell_{ \rho_{n}}\left(\Lambda_{nI}\right) + \epsilon +
\frac{2A_{\tau}\sqext(K_{\tau})}{n}\\
&= \ell_{\rho_{n}}(\mu) +\epsilon + \frac{2A_{\tau}\sqext(K_{\tau})}{n}\\
&\leq \sup_{\rho}\frac{\ell_{\rho}(\mu)}{\sqrt{\mathrm{Area}(\rho)}} + \epsilon +\frac{2A_{\tau}\sqext(K_{\tau})}{n}.
\end{align}
As $n$ can be taken arbitrarily large and $\epsilon$ can be arbitrarily small, we have
$$\sqext(\mu)\leq \sup_{\rho}\frac{\ell_{\rho}(\mu)}{\sqrt{\mathrm{Area}(\rho)}}$$
Putting together with Proposition \ref{prop.lower}, we complete the proof of 
Theorem \ref{thm.main}. \qed

%
%

\appendix
\section{Diameter of extremal metrics}\label{sec.extremal metric}
Given a weighted multi-curve $c$ and $X\in\T$, a conformal metric $\rho$ is called an {\em extremal metric} if
$$
\ext(c) = \frac{\ell_{\rho}(c)^{2}}{\mathrm{Area}(\rho)}.
$$
In other words, the supremum defining the extremal length is attained by the metric  $\rho$.

When $c$ is simple, an extremal metric is given by a quadratic differential \cite{Jenkins}, however very little is known about extremal metrics of non-simple curves, see \cite[Section 4.8]{GT}, and \cite{HZ} and references therein.
To the best of augthor's knowlege, alghtough in \cite[Theorem 12]{Burton}, the existence of so called ``general extremal metrics'' is proved, the existence of extremal metrics is unknown for non-simple case.
In this appendix, we assume the existence of such extremal metrics, and give an upper bound on the diameter of these extremal metrics whose area is normalized to be $1$ (see Theorem \ref{thm.extbound}).

We first observe that the extremality for conformal metrics
is necessary to have a bounded diameter.
\begin{ex}[Hyperbolic punctured disc]
Consider $B(0,1/e)$, the ball of radius $1/e$ with center $0$.
The conformal metric 
$$
\rho(z):=\frac{1}{|z|\log(1/|z|)}
$$
has $\mathrm{Area}(\rho) = 2\pi$ and infinite diameter.
The length of the boundary equals $2\pi e$.
These can be verified by integrating in polar coordinates and making the change of variables \( r = e^{-u} \).
This example shows that, in general, conformal metrics may have infinite diameters even when the area is bounded.
\end{ex}

\subsection{Properties of extremal metrics}
From now on, we fix $X\in\T$.
\begin{lem}\label{lem.extconf}
Suppose that $\rho$ is an extremal metric for some weighted multi-curve $c$.
Then for any $p\in X$ and any simply connected open neighborhood $p\in U\subset X$, and for any $\epsilon>0$ we have the following:
\begin{itemize}
\item[($\ast$)] there must exist a representative $\gamma$ of some connected component of $c$ such that $\gamma\cap U\neq\emptyset$ and $L_{\rho}(\gamma)\leq \ell_{\rho}(\gamma)+\epsilon$.
\end{itemize}
\end{lem}
\begin{proof}
Fix $\epsilon>0$. Suppose contrary that there exists a simply connected open neighborhood $p\in V\subset X$ disjoint from any representative $\gamma$ of $c$ with
$L_{\rho}(\gamma)\leq \ell_{\rho}(\gamma)+\epsilon/2$.
We further suppose that $V$ is maximal among such open neighborhoods.
Then, the dense subset of the closure $\bar V$ of $V$ must intersect with some
representative $\gamma$ of $c$ with $L_{\rho}(\gamma)\leq \ell_{\rho}(\gamma)+\epsilon/2$.

First, we suppose that $V$ has $\rho$-area zero.
Then, as the extremal length of curves that touch the boundary and enclose $p$
is finite \cite[Section 4.8]{Ahlfors}, one may suppose that $p$ is contained in a neighborhood $V'\subset V$ such that $L_{\rho}(\partial V')=0$.
By identifying $V'$ with a rectangle, we see that a dense subset of horizontal lines must have length less than $\epsilon/2$ ($V'$ has area zero).
Hence, the $\rho$-distance from any neighborhood of $p\in V''\subset V$ to the boundary $\partial V$ is less than $\epsilon/2$.
Therefore, any neighborhood of $p$ must intersect with some
representative $\gamma$ of $c$ with $L_{\rho}(\gamma)\leq \ell_{\rho}(\gamma)+\epsilon$.

Now, we suppose that $V$ has a positive area.
Then there exists $A\subset V$ with positive area so that any representative $\gamma'$ of any component of $c$ with $\gamma'\cap A\neq \emptyset$ satisfies
$L_{\rho}(\gamma')> \ell_{\rho}(\gamma')+\epsilon/2$.
In this case, we can reduce $\mathrm{Area}(\rho)$ by assigning slightly smaller values in $A$ without changing $\ell_{\rho}(c)$.
This contradicts the assumption that $\rho$ is extremal.
\end{proof}

The following corollary, while independent of the proof of Theorem \ref{thm.extbound},
 is nevertheless worth noting.
 One compares Corollary \ref{coro.hyp} with Remark \ref{rmk.hyp}.
\begin{coro}\label{coro.hyp}
The hyperbolic metric can never be an extremal metric for any weighted multi-curve.
\end{coro}
\begin{proof}
Let $c$ be a weighted multi-curve.
The hyperbolic geodesic representative of any free homotopy class of closed curves is unique.
Hence, we always have an open neighborhood $V$ that does not intersect with small neighborhoods of those geodesic representatives of $c$ for small enough $\epsilon>0$. Hence, Lemma \ref{lem.extconf} and Morse Lemma (quasi-geodesics are contained in a neighborhood of geodesics) imply that the hyperbolic metric is not extremal.
\end{proof}

\begin{coro}\label{coro.key}
Let $\rho$ be an extremal metric for a weighted multi-curve $c$.
Suppose that there is a simply connected region $C\subset X$ with $L_{\rho}(\partial C)<\infty$.
Then for any $p\in C$, we have
$$
d_{\rho}(p,\partial C)\leq L_{\rho}(\partial C)/4.
$$
In particular, we have
$$
\mathrm{diam}(C)\leq L_{\rho}(\partial C).
$$
\end{coro}
\begin{proof}
Let $\epsilon>0$ be arbitrarily small.
By Lemma \ref{lem.extconf}, there exists a representative \( \gamma \) of \( c \) that passes arbitrarily close to \( p \) with
\begin{equation}
L_{\rho}(\gamma)\leq \ell_{\rho}(\gamma)+\epsilon. \label{eq.**}
\end{equation}

Consider the connected component \( \gamma' \) of \( \gamma \cap C \) that is closest to \( p \), and let \( a, b \) denote the points \( \gamma' \cap \partial C \).  
By Equation \eqref{eq.**}, we have
\(
L_{\rho}(\gamma') \leq d_{\rho}(a,b) + \epsilon,
\)
and it is clear that
\(
d_{\rho}(a,b) \leq {L_{\rho}(\partial C)}/{2}.
\)
Moreover, since \( \gamma' \) can be chosen arbitrarily close to \( p \), we obtain
\[
2d_{\rho}(p, \partial C) \leq L_{\rho}(\gamma') + \epsilon.
\]

Therefore we have
$$
2d_{\rho}(p,\partial C) \leq L_{\rho}(\gamma')+  \epsilon\leq d_{\rho}(a,b)+2\epsilon\leq L_{\rho}(\partial C)/2 + 2\epsilon.
$$

Since \( \epsilon > 0 \) can be chosen arbitrarily small, we conclude that
\[
d_{\rho}(p, \partial C) \leq {L_{\rho}(\partial C)}/{4}.
\]

Given any two points $p,q\in C$, we have $p_{C},q_{C}\in\partial C$ such that
$$d_{\rho}(p,p_{C}), d_{\rho}(q,q_{C})\leq L_{\rho}(\partial C)/4.$$
Then we can connect $p_{C}, q_{C}$ along $\partial C$ with length at most $L_{\rho}(\partial C)/2$.
Hence we have $d_{\rho}(p,q)\leq d_{\rho}(p,p_{C})+d_{\rho}(p_{C},q_{C})+d_{\rho}(q,q_{C})\leq L_{\rho}(\partial C)$.
\end{proof}

Now consider a pants decomposition $\Pi$ of $X$.
We suppose that all the cuffs of pairs of pants in $\Pi$ are hyperbolic geodesics and hence determined solely by $X\in\T$.
\begin{defi}
Let $\{\Gamma_{1},\cdots, \Gamma_{3g-3}\}$ denote the family of closed geodesics which are cuffs of $\Pi$, and
$P_{1},\dots, P_{2g-2}$ denote the set of pairs of pants in $\Pi$.
Let $P_{ij}^{k}$ be the hyperbolic surface obtained by gluing $P_{i}$ and $P_{j}$ along the cuff $\Gamma_{k}$. If $P_{i}$ and $P_{j}$ do not share $\Gamma_{k}$ 
we set $P_{ij}^{k} = P_{i}\sqcup P_{j}$.
We denote by $\delta^{k}_{l,m}$ the shortest hyperbolic geodesic connecting $\Gamma_{l} \text{ and }\Gamma_{m}$ in $P_{i,j}^{k}$.
Then, let $\Gamma^{k}_{l,m}$ denote the family of arcs homotopic to $\delta^{k}_{l,m}$ relative to $\Gamma_{l} \text{ and }\Gamma_{m}$.
We define $D'$ by
$$
D':=\max\{\mathrm{Ext}_{P_{i}}(\Gamma_{j}), 
\mathrm{Ext}_{P_{ij}^{k}}(\Gamma^{k}_{l,m})\}
$$
where 
\begin{itemize}
\item
if $\Gamma_{j}$ is not a cuff of $P_{i}$, we set $\mathrm{Ext}_{P_{i}}(\Gamma_{j})=0$.
\item 
If $\Gamma^{k}_{l,m}$ does not have arcs contained in $P_{ij}^{k}$, we set $\mathrm{Ext}_{P_{ij}^{k}}(\Gamma^{k}_{l,m})=0$
(Note that the extremal length of seams in each $P_{i}$ is also taken into consideration here).
\end{itemize}
\end{defi}
The constant $D'$ depends only on $X$, as we consider curve families and pairs of pants determined only by the hyperbolic geometry of $X$.

\begin{proof}[Proof of Theorem \ref{thm.extbound}.]
Let $\rho$ be a conformal metric of area $1$ on $X$.
Let us first fix a pair of pants $P_{i}$ from $\Pi$.
By the definition of extremal length and of $D'$, there is a curve $\gamma_{j}\subset P$ homotopic to $\Gamma_{j}$ such that
\begin{align}
&\frac{L_{\rho}(\gamma_{j})^{2}}{\mathrm{Area}(\rho|P_{i})}\leq D'\Longrightarrow L_{\rho}(\gamma_{j})\leq \sqrt{D'}\label{eq.ext1}
\end{align}
Similarly, there is a curve $ \gamma_{\ell,m}^{k}\subset P_{i,j}^{k}$ homotopic to an arc in $\Gamma^{k}_{l,m}$ relative to the boundary such that
\begin{align}
&\frac{L_{\rho}( \gamma_{\ell,m}^{k})^{2}}{\mathrm{Area}(\rho|P_{i,j}^{k})}\leq D'\Longrightarrow L_{\rho}(\gamma_{\ell,m}^{k})\leq \sqrt{D'}\label{eq.ext2}
\end{align}

We choose $\gamma_{i}$ and $\gamma_{\ell,m}^{k}$ to be $\rho$-geodesics.
Then, we will decompose $X$ by using curves $\gamma_{i}$ and $ \gamma_{\ell,m}^{k}$ (Note that these curves depend on $\rho$).

\begin{figure}[ht]
\includegraphics[width=5cm]{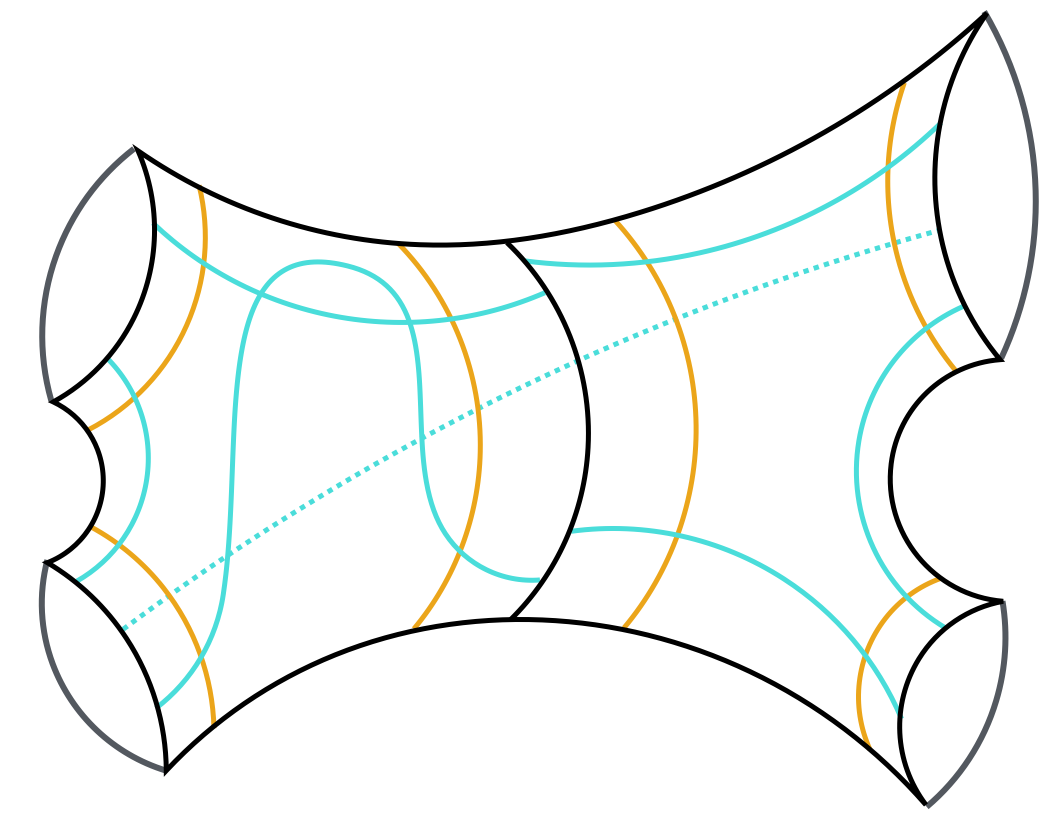}
\caption{The cuffs $\gamma_{i}$'s and seams $ \gamma_{\ell,m}^{k}$'s on $X$.}
\label{fig.decomp}
\end{figure}

All the situations we must consider are depicted in Figure \ref{fig.decomp}.
As we see in the pair of pants on the left side of Figure \ref{fig.decomp}, 
cuffs $\gamma_{i}$'s and seams $ \gamma_{\ell,m}^{k}$'s can intersect each other.
Nevertheless, the curves $\gamma_{i}$'s and arcs $ \gamma_{\ell,m}^{k}$'s decompose the complement of $\gamma_{i}$'s to simply connected regions.
On each pair of pants, the total length of boundaries of those simply connected regions is bounded from above by $9\sqrt{D'}$ (cuffs are used once, seams are used twice, 
and all the cuffs and seams have length $<\sqrt{D'}$).
Then by Corollary \ref{coro.key}, we see that if $\rho$ is extremal, the diameter is bounded from above by $9\sqrt{D'}$.

Similarly, we may cut each annular region around the cuffs of $P_{i}$ 's (in the middle of Figure \ref{fig.decomp}) by a path of length less than $\sqrt{D'}$ (dotted line is one of $\gamma_{\ell,m}^{k}$'s).
Hence, this annular region is decomposed into simply connected regions whose total perimeter, and therefore the diameter, are bounded above by \( 4\sqrt{D'} \).

There are $2g-2$ pairs of pants and $3g-3$ annuli.
Therefore, we have
\begin{equation}
\mathrm{diam}_{\rho}(X)\leq (2g-2)\cdot 9\sqrt{D'} + (3g-3)\cdot 4\sqrt {D'} = 30(g-1)\sqrt{D'}.\label{eq.diam}
\end{equation}
\end{proof}

\section*{Acknowledgement}
The author would like to thank ChatGPT, Chris Leininger, D\'idac Mart\'inez-Granado, Sadayoshi Kojima, Greg McShane, Ryo Matsuda, Hideki Miyachi, and Toshiyuki Sugawa for their helpful conversations.
The author also thanks the anonymous referee for the careful reading and useful suggestions.
\bibliographystyle{alpha} 
\bibliography{references} 
\end{document}